\documentclass[12pt,hyperref]{article}
\usepackage{graphicx} 
\usepackage{float} 
\usepackage{amssymb}
\usepackage{amsmath}
\usepackage{mathtools}
\usepackage[thmmarks,amsmath]{ntheorem} 
\usepackage{lineno}
\usepackage[numbers,sort&compress]{natbib}
\usepackage{bm} 
\usepackage{extarrows}
\usepackage{mathrsfs} 
\usepackage{booktabs} 
{
    \theoremstyle{nonumberplain}
    \theoremheaderfont{\bfseries}
    \theorembodyfont{\normalfont}
    \theoremsymbol{\mbox{$\Box$}}
    \newtheorem{proof}{Proof}
    \newtheorem{Sketch of proof}{Sketch of proof}
}

\usepackage{theorem}
\newtheorem{theorem}{Theorem}[section]

\newtheorem{lemma}{Lemma}[section]
\newtheorem{question}{Question}[section]

\newtheorem{corollary}{Corollary}[section]

\newtheorem{claim}{Claim}[section]

\newtheorem{conjecture}{Conjecture}[section]
{
    \theoremheaderfont{\bfseries}
    \theorembodyfont{\normalfont}
    
}
\newcommand{\RNum}[1]{\uppercase\expandafter{\romannumeral #1\relax}}

\graphicspath{{figure/}}                                 
\usepackage[a4paper]{geometry}                        
\begin{document}
\title{\bf  Removal paths avoiding vertices\footnote{The author's work is supported by NNSF of China (No.12071260).}}
\date{}
\author{\sffamily Yuzhen Qi,Jin Yan\footnote{Corresponding author. E-mail address: yanj@sdu.edu.cn.}\\
    {\sffamily\small School of Mathematics, Shandong University, Jinan 250100, China }}
\maketitle

{\noindent\small{\bf Abstract:}
In this paper, we show that for any positive integer $m$ and $k\in [2]$, let $G$ be a $(2m+2k+2)$-connected graph and let $a_1,\ldots , a_m, s, t$ be any distinct vertices
of $G$, there are $k$ internally disjoint $s$-$t$ paths $P_1, \ldots, P_k$ in $G$ such that $\{a_1,\ldots , a_m\} \cap \bigcup^{k}_{i=1}V (P_i) = \emptyset$ and $G- \bigcup^{k}_{i=1}V (P_i)$ is 2-connected, which  generalizes
the result  by Chen, Gould and Yu [Combinatorica 23 (2003) 185--203], and Kriesell [J. Graph Theory 36 (2001) 52--58]. The case $k=1$ implies that for any $(2m+5)$-connected graph $G$, any edge $e \in E(G)$, and any distinct vertices $a_1,\ldots , a_m$ of $G-V(e)$, there exists a cycle $C$ in $G- \{a_1,\ldots , a_m\}$ such that $e\in E(C)$ and $G- V(C)$ is 2-connected, which improves the bound $10m+11$ of Y. Hong, L. Kang and X. Yu in [J. Graph Theory 80 (2015) 253--267].

 \vspace{1ex}

\vspace{1ex}
{\noindent\small{\bf AMS subject classifications. 05C38; 05C40; 05C75} 

\section{Introduction}
We discuss only finite simple graphs and use standard terminology and notation from \cite{B1} except as indicated.  A path $P$ with end vertices $u, v$ is called an $u$-$v$ path, and for two vertices $x, y \in V(P)$, we denote the subpath of $P$ from $x$ to $y$ by $P[x, y]$, and  let
$P(x, y) = P[x, y]-\{x, y\}$. The following conjecture is a generalization of an old conjecture made by  Lov\'{a}sz \cite{L} in 1975.

\begin{conjecture}\label{conjecture2}
 \cite{K} There exists a function $f = f (k,l)$ such that the following holds. For every $f (k,l)$-connected
graph $G$ and two distinct vertices $s$ and $t$ in $G$, there are $k$ internally disjoint $s$-$t$ paths $P_1, \ldots, P_k$ such that $G-\bigcup^{k}_{i=1}V (P_i)$ is $l$-connected.
\end{conjecture}

A famous result of Tutte \cite{Tutte} yields $f(1,1) = 3$. The case $l=2$ was independently obtained by \cite{Chen} and \cite{Kriesell}, who showed $f(1,2) = 5$.  It is shown in  \cite{K} that $f(k, 1)= 2k + 1$  and $f(k, 2)\leq  3k+2$. Conjecture \ref{conjecture2} is still (wide) open for $f(k,l\geq  3)$.

Indeed, Tutte \cite{Tutte} proved a stronger reault. That is, for any 3-connected graph $G$ and distinct vertices $a$, $s$, $t$ of $G$, there is an $s$-$t$ path $P$ in $G$ such that $a \notin V(P)$ and $G-V(P)$ is connected. A result of Jung \cite{Jung} implies that for any 6-connected graph $G$ and any distinct vertices $a_1$, $a_2$, $s$, $t$ of $G$, there is an $s$-$t$ path $P$ in $G$ such that $a_1, a_2 \notin V (P)$ and
$G-V(P)$ is connected. These results motivate us to propose the following question.
\begin{question}\label{question2}
Let $m,k,l$ be positive integers. There exists a function $f(m,k,l)$ such that the following holds. For every
 $f(m,k,l)$-connected
graph $G$ and any distinct vertices $a_1,\ldots , a_m, s, t$
of $G$, there are $k$ internally disjoint $s$-$t$ paths $P_1, \ldots, P_k$ in $G$ such that $\{a_1,\ldots , a_m\} \cap \bigcup^{k}_{i=1}V (P_i) = \emptyset$ and $G- \bigcup^{k}_{i=1}V (P_i)$ is $l$-connected.
\end{question}

The result in \cite{Ma} yields $f(m,k,2)\leq 10m+30k+2$. This function may not be optimal since authors in \cite{Ma} used the result that $10k$-connected graph is $k$-linked in \cite{Thomas}, and $10k$ is not to optimal for the $k$-linkage problem. Recently, X. Du, Y. Li , S. Xie and X. Yu \cite{Du} verified Question \ref{conjecture2} in the $k=l=1$ case by  proving the following result.


 \begin{theorem}\label{theorem2}
\cite{Du} Let $m\geq 1$ be an integer and let $G$ be a $(2m+2)$-connected graph. For any distinct vertices $a_1,\ldots , a_m, s, t$ of $G$, there is an $s$-$t$ path $P$ in $G$ such that $\{a_1,\ldots , a_m\} \cap V(P) = \emptyset$ and $G-V(P)$ is connected.
  \end{theorem}

%

Our first result gives some progress towards Question \ref{question2} when $l=2$  and $k=1,2$.
 \begin{theorem}\label{theorem7}
For any positive integer $m$ and $k=1,2$, we have $f(m,k,2)\leq 2m+2k+2$.
\end{theorem}

%
%
%


Lov\'{a}sz's conjecture \cite{L} can also be phrased in terms of finding a cycle containing an arbitrary edge $e$ such that deleting the vertices of the cycle leaves the graph $l$-connected. Y. Hong, L. Kang and X. Yu \cite{Hong} thought that in potential applications one may need the cycle to avoid certain vertices, and posed the following conjecture.

\begin{conjecture}\label{conjecture1}
\cite{Hong} For any positive integers $m$ and $l$, there exists a smallest
positive integer $g(m, l)$ such that for any $g(m, l)$-connected graph $G$, any edge $e \in E(G)$, and any distinct vertices $a_1,\ldots , a_m$ of $G-V(e)$, there exists an induced cycle $C$ in $G- \{a_1,\ldots , a_m\}$
such that $e\in E(C)$ and $G- V(C)$ is $l$-connected.
\end{conjecture}

Theorem \ref{theorem2} implies $g(m,1)\leq 2m+3$ (set $e
=st$, and then $G-V(e)$ is  $(2m+2)$-connected). Our Theorem 1.3  establishes Conjecture \ref{conjecture1} for $l=2$, improves the bound $10m+11$ in \cite{Hong}.

\begin{theorem}
For any positive integer $m$, $g(m,2)\leq 2m+5$.
\end{theorem}

As a natural generalization of   Conjecture \ref{conjecture1}, the last result in our paper is the following:

\begin{theorem}\label{theorem3}
Let $m$ be a positive integer and let $G$ be a $(2m+8)$-connected graph. For two distinct edges $\{e_1,e_2\}\subseteq E(G)$, and any distinct vertices $a_1,\ldots , a_m$ of $G-V(\{e_1,e_2\})$, there exists a cycle $C$ in $G- \{a_1,\ldots , a_m\}$ such that $\{e_1,e_2\}\subseteq E(C)$ and $G-V(C)$ is connected.
\end{theorem}

The rest of the paper is organized as follows. In Section \ref{section2}, we introduce notation, and we give a characterization of ``$m$-rooted graphs'' in Theorem \ref{theorem8}, which will be used in the proof of Theorems  \ref{theorem7} and \ref{theorem3}.  So, the proof of Theorem \ref{theorem8} is given in Section \ref{section3}. Theorems  \ref{theorem7} and \ref{theorem3} are proved in Section \ref{section4}.

\section{Preliminaries and tools}\label{section2}
\subsection{Basic terminology}

 For an integer $k$, we write $[k]$ for $\{1,\ldots,k\}$.  A \emph{block} of a graph $G$ is a maximal 2-connected subgraph of $G$ or an isolate vertex. For a graph $G$, we use $V(G)$, $E(G)$, $v(G)$, and $e(G)$ to denote its vertex set, edge set, the cardinality of its vertex set, and the cardinality of its edge set, respectively. For any $S\subseteq V(G)$, the subgraph of $G$ induced by $S$ is denoted by $G[S]$. Let $G-S =G[V\backslash S]$ and $N_G[S]=N_G(S)\cup S$.


We follow the notation of \cite{Du}.  Let $m\geq 0$ be an integer, and we call $(G, \{a_1,\ldots , a_m\}, b_1,b_2)$  an $m$-\emph{rooted  graph} if $G$ is a graph and $a_1,\ldots , a_m, b_1, b_2$ are distinct vertices of $G$ (note that $m=0$ means that  $\{a_1,\ldots , a_m\}=\emptyset$). An $m$-rooted graph $(G, \{a_1,\ldots , a_m\}, b_1, b_2)$ is said to be \emph{feasible} (resp. 2-\emph{feasible}) if $G$ contains a $b_1$-$b_2$ path $P$ and $G-V(P)$ has a component (resp. block) containing  $\{a_1,\ldots , a_m\}$. Further,  we call an $m$-rooted graph $(G, \{a_1,\ldots , a_m\}, b_1, b_2)$  (2,2)-\emph{feasible} if $\{a_1,\ldots , a_m\}$ and $\{b_1,b_2\}$ are contained in two disjoint blocks, respectively.
\begin{figure}[H]
	\centering
	\includegraphics[scale=0.28]{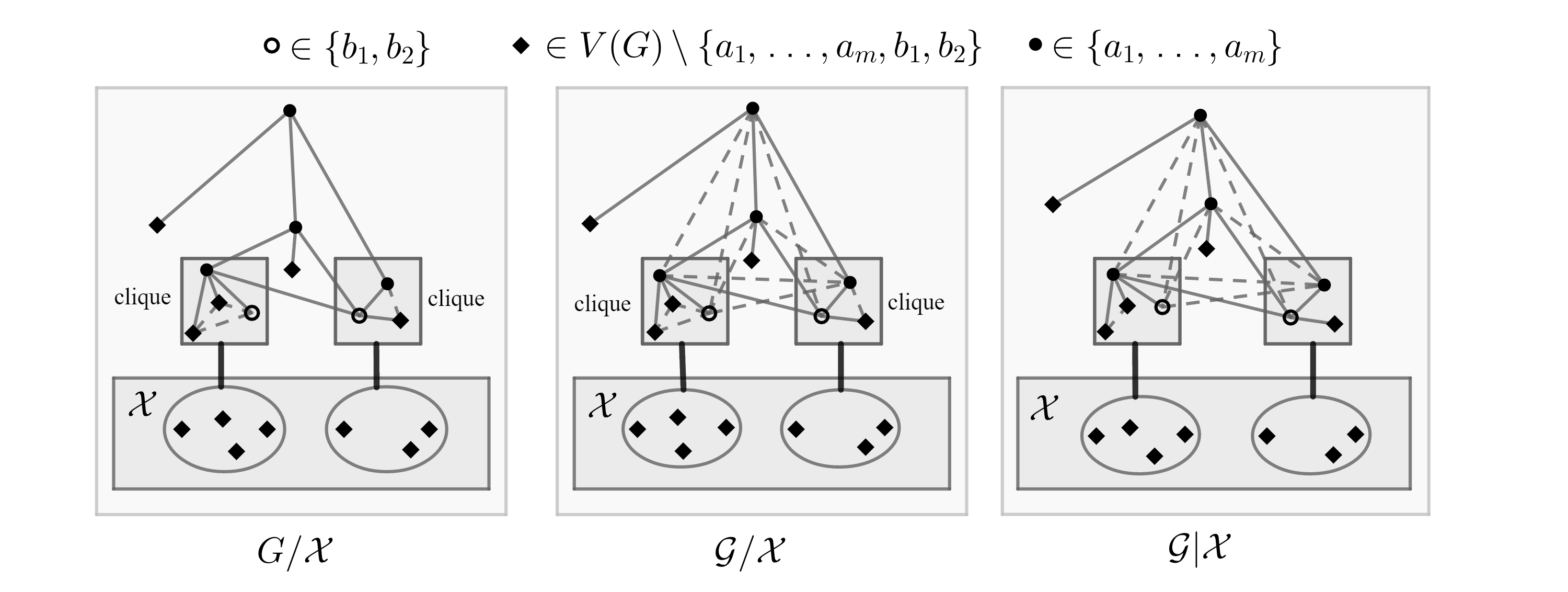}
	\caption{An example of $G/\mathcal{X}$, $\mathcal{G}/\mathcal{X}$ and $\mathcal{G}|\mathcal{X}$. Dashed lines indicate edges that do not belong to $E(G)$, but belong to  $E(G/\mathcal{X})$ or $E(\mathcal{G}/\mathcal{X})$ or $E(\mathcal{G}|\mathcal{X})$.}
	\label{fig1}
\end{figure}

Let $G$ be a graph and $\{s_1,\ldots , s_t\}\subseteq V(G)$, we say that $(G, s_1,\ldots , s_t)$  is \emph{planar} if $G$ can be drawn in a closed disc in the plane with no edge crossings such that $s_1,\ldots , s_t$ on the boundary of the disc in clockwise order. For any $S\subseteq V(G)$, an $\mathcal{S}$-\emph{collection} in $G$ is a collection $\mathcal{X}\subseteq V(G)\setminus S$ such that $N_G[X_1]\cap X_2 =\emptyset$ for distinct $X_1, X_2 \in \mathcal{X}$. We use $G/\mathcal{X}$ to denote the graph obtained from $G$ by, for each $X \in \mathcal{X}$, deleting $X$ and adding edges such that $G[N_G(X)]$ is a clique. For an $m$-rooted  graph $\mathcal{G}=(G, \{a_1,\ldots , a_m\}, b_1, b_2)$ and any $\{a_1,\ldots , a_m, b_1, b_2\}$-collection $\mathcal{X}$ in
$G$, let $\mathcal{G}/\mathcal{X}$ denote the graph obtained from $G/\mathcal{X}$ by adding an edge joining every pair of
distinct vertices in $\{a_1,\ldots , a_m, b_1, b_2\}$ except the pair in $\{b_1, b_2\}$, and let $\mathcal{G}|\mathcal{X}$ denote the graph obtained from $\mathcal{G}/\mathcal{X}$ by deleting all edges in
$\{uv\in E(\mathcal{G}/\mathcal{X})\setminus E(G): u\in V(G)\setminus \{a_1,\ldots , a_m, b_1, b_2\}, v\in \{a_1,\ldots , a_m, b_1, b_2\}\}$. An example of $G/\mathcal{X}$, $\mathcal{G}/\mathcal{X}$ and $\mathcal{G}|\mathcal{X}$ is shown in Fig. \ref{fig1}. Note that
\begin{center}
  $e(\mathcal{G}|\mathcal{X})\leq e(\mathcal{G}/\mathcal{X})$ and   $e(G/
  \mathcal{X})\leq e(\mathcal{G}/\mathcal{X})$.
\end{center}


\subsection{$m$-rooted graphs}
The following result conducted by Seymour \cite{P} gave a characterization of feasible 2-rooted graphs.

 \begin{theorem}\label{theorem1}
\cite{P} Let $(G, \{a_1,a_2\}, b_1, b_2)$ be a $2$-rooted graph. Then either $(G, \{a_1, a_2\}, b_1, b_2)$ is feasible, or there is some $\{a_1, a_2, b_1, b_2\}$-collection $\mathcal{X}$ in $G$ such that $|N_G(X)| \leq 3$ for all $X\in \mathcal{X}$ and $(G/\mathcal{X} , a_1, b_1, a_2, b_2)$ is planar.
\end{theorem}

The natural generalization of Theorem \ref{theorem1} to  $m$-rooted graphs,  proved by  X. Du, Y. Li, S. Xie and X. Yu \cite{Du} in 2023,  is the following:

 \begin{theorem}
\cite{Du} Let $m\geq 0$ be an integer, and let $(G, \{a_1,\ldots , a_m\}, b_1, b_2)$ be an $m$-rooted graph. Then either $(G, \{a_1,\ldots , a_m\}, b_1, b_2)$ is feasible, or there is some $\{a_1,\ldots , a_m, b_1,b_2\}$-collection $\mathcal{X}$ in $G$ such that $|N_G(X)| \leq m+1$ for all $X\in \mathcal{X}$, and
 $e(\mathcal{G/X})\leq (m+1)v(G/\mathcal{X})-m^2/2-3m/2-1$.
\end{theorem}

%

 Now we are able to give a characterization of $m$-rooted graphs, which will be used in the proof of Theorems \ref{theorem7} and \ref{theorem3}.


%
%
%

\begin{theorem}\label{theorem8}
Let $m\geq 0$ be an integer, and let $(G, \{a_1,\ldots,a_m\}, b_1,b_2)$ be an $m$-rooted graph of order at least $m+3$. Suppose that there is no $\{a_1,\ldots,a_m, b_1,b_2\}$-collection $\mathcal{X}$ in $G$ such that $|N_G(X)| \leq m+2$ for all $X\in \mathcal{X}$, and $e(\mathcal{G}| \mathcal{X})\leq (m+2)v(\mathcal{G}|\mathcal{X})-m^2/2-5m/2-4$. Then

(i) either $a_1,\ldots,a_m, b_1,b_2$ have a common neighbour in $G-\{a_1,\ldots,a_m, b_1,b_2\}$, or $(G, \{a_1,\ldots,$\\$a_m\}, b_1,b_2)$ is 2-feasible.

(ii) either $a_1,\ldots,a_m$ have a common neighbour in $G-\{a_1,\ldots,a_m, b_1,b_2\}$, or $(G, \{a_1,\ldots,a_m\},$\\$ b_1,b_2)$ is (2,2)-feasible.
\end{theorem}

%

Due to Theorem \ref{theorem8}, we have the following results.

\begin{corollary} \label{corollary1}
Let $m\geq 0$ be an integer, and let $(G, \{a_1,\ldots , a_m\}, b_1, b_2)$ be an $m$-rooted graph. Then the following statements hold.

(i) If $G$ is $(2m+4)$-connected, then $(G, \{a_1,\ldots , a_m\}, b_1,b_2)$ is 2-feasible.

(ii) If $G$ is $(2m+5)$-connected, then $(G, \{a_1,\ldots , a_m\}, b_1,b_2)$ is (2,2)-feasible.
\end{corollary}
\begin{proof}
To prove this, we first show that if $(G, \{a_1,\ldots,a_m\}, b_1,b_2)$ is $(2m+4)$-connected, then $G$ satisfies the condition of Theorem \ref{theorem8}. That is, there is no $\{a_1,\ldots , a_m,b_1, b_2\}$-collection $\mathcal{X}$ in $G$ such that  $|N_G(X)| \leq m+2$ for all $X\in \mathcal{X}$, and $e(\mathcal{G}| \mathcal{X})\leq (m+2)v(\mathcal{G}|\mathcal{X})-m^2/2-5m/2-4$.
Suppose not, then  $\mathcal{X}=\emptyset$ or $\mathcal{X}=V(G)\setminus \{a_1,\ldots , a_m, b_1, b_2\}$ by the connectivity of $G$. This yields
\begin{equation*}
e(\mathcal{G}| \mathcal{X})\geq\left\{
\begin{aligned}
&(m+2) v(\mathcal{G}| \mathcal{X}),\ \text{if}\ \mathcal{X}=\emptyset, \\
&(m+2)(m+1)/2,\ \text{if}\ \mathcal{X}=V(G)\setminus \{a_1,\ldots , a_m, b_1, b_2\}.
\end{aligned}
\right.
\end{equation*}
In each case, $e(\mathcal{G}| \mathcal{X})> (m+2)v(\mathcal{G}| \mathcal{X})-m^2/2-5m/2-4$, a contradiction. So if $G$ is $(2m+4)$-connected, then $G$ satisfies the condition of Theorem \ref{theorem8}.

 (i) If not, then $a_1,\ldots , a_m, b_1,b_2$ have a common neighbour in $G-\{a_1,\ldots , a_m, b_1,b_2\}$ by Theorem \ref{theorem8} (i). So $G-\{a_1,\ldots , a_m\}$ has a $b_1$-$b_2$ path $P$ of order three, and then $G-V(P)$ is $(2m+1)$-connected. This implies $(G, \{a_1,\ldots , a_m\}, b_1,b_2)$ is 2-feasible, a contradiction. Thus (i) holds.

(ii) If not, then $a_1,\ldots , a_m$ have a common neighbour $v$ in $G-\{a_1,\ldots , a_m, b_1, b_2\}$ by Theorem \ref{theorem8} (ii). Let $G^\prime=G-\{v\}$. Then $G^\prime$ is $(2m+4)$-connected. By Theorem \ref{theorem8} (ii) again, $a_1,\ldots , a_m$ still have a common neighbour in $G^\prime-\{a_1,\ldots , a_m, b_1, b_2\}$. This implies that $G$ has a block $B$ of size $m+2$ containing  $\{a_1,\ldots , a_m\}$. Since $G$ is $(2m+5)$-connected, $G-V(B)$ is $(m+3)$-connected, and then $(G, \{a_1,\ldots , a_m\}, b_1,b_2)$ is (2,2)-feasible, which contradicts our assumption. Thus (ii) holds.
\end{proof}

\begin{corollary} \label{corollary2}
 Let $m$ be a positive integer and let $G$ be a $(2m+8)$-connected graph. For two distinct edges $\{e_1,e_2\}\subseteq E(G)$, and any distinct vertices $a_1,\ldots , a_m$ of $G-V(\{e_1,e_2\})$, there exists a cycle $C$ in $G$ such that $\{e_1,e_2\}\subseteq E(C)$ and $G-V(C)$ has a component containing  $\{a_1,\ldots , a_m\}$.
\end{corollary}
\begin{proof}
On the contrary, suppose this lemma fails. Write $e_1=b_1b_2$, $e_2=b_3b_4$, $A= \{a_1,\ldots , a_m\}$ and $B=\{b_1,b_2,b_3,b_4\}$. Since $G$ is $(2m+8)$-connected, there exist four distinct vertices, say $u_1,u_2,u_3,u_4$, in  $G- (A\cup B)$ such that $u_j\in N_{G} (b_j)$ for each $j\in [4]$.  For any $i\in \{1,3\}$, let $G^{\prime}$ be the graph obtained from $G$ by deleting all edges between $b_i,b_{i+1}$ and $N_{G}(b_i)\cup N_{G}(b_{i+1})  \cup \{u_i,u_{i+1}\} \setminus (N_{G}(b_{i})\cap N_{G}(b_{i+1}))$, and contracting $b_i,b_{i+1}$ to a vertex $b_i^\prime$. Then the following results hold.

\begin{claim}\label{claim4}
  $G^{\prime}-\{b_1^\prime,b_3^\prime\}$ is $(2m+4)$-connected, and for any $i\in \{1,3\}$, we have $d_{G^\prime-A}(b_i^\prime)\geq 2$.
 \end{claim}

 \begin{claim}\label{claim5}
 For any $x\in \{b_1,b_2 \}$ and $y\in \{b_3,b_4\}$, $x,y$  have no  common neighbour in  $G- (A\cup B)$.
 \end{claim}
\begin{proof} Suppose for  a contradiction that there exists $v\in G- (A\cup B)$ such that $xv,yv\in E(G)$. One may assume $x=b_1$ and $y=b_3$. Let $G^{\prime}=G-\{v,b_1,b_3\}$. Then  $(G^{\prime}, A,b_2,b_4)$ is $(2m+5)$-connected. So, this lemma follows from Corollary \ref{corollary1} (i) and our assumption.
\end{proof}

 \begin{claim}\label{claim6}
If $G^{\prime}$ (resp. $G$) has a component $H$  containing $A$, then $G^{\prime}-V(H)$ (resp. $G-V(H)$) has no block containing $\{b_1^\prime,b_3^\prime\}$ (resp. $B$). That is, $(G^{\prime}, A, b_1^\prime,b_3^\prime)$ is not 2-feasible.
 \end{claim}
\begin{proof}
Suppose not, then by Menger's Theorem and our construction, $G-V(H)$ has a cycle $C$ such that $\{e_1,e_2\}\subseteq E(C)$, which contradicts our assumption. Hence this claim holds.
\end{proof}

Let $\mathcal{G^{\prime}}=(G^{\prime}, A, b_1^\prime,b_3^\prime)$. Combining Claim \ref{claim6}  with  Theorem \ref{theorem8} (i) and the connectivity of $G$ implies that there is a $A\cup \{b_1^\prime,b_3^\prime\}$-collection $\mathcal{X}$ in $G^{\prime}$ such that $|N_{G^{\prime}}(X)| \leq m+2$ for all $X\in \mathcal{X}$, and $e(\mathcal{G^{\prime}}| \mathcal{X})\leq (m+2)v(\mathcal{G^{\prime}}|\mathcal{X})-m^2/2-5m/2-4$. By Claim \ref{claim4} ($G^{\prime}-\{b_1^\prime,b_3^\prime\}$ is $(2m+4)$-connected), $\mathcal{X}=\emptyset$, or $V(G^{\prime})\setminus \{A, b_1^\prime, b_3^\prime,u,v\}\subseteq \mathcal{X}$  with $\{u,v\} \subseteq V(G^{\prime})\setminus \{A, b_1^\prime, b_3^\prime\}$. Further, Claims \ref{claim4} and \ref{claim5} state that $d_{G^{\prime}-\{A, b_1^\prime, b_3^\prime,u,v\}}(b_1^\prime \cup b_3^\prime)\geq 2$. So $\mathcal{X}=\emptyset$ or $\mathcal{X}=V(G^{\prime})\setminus \{A, b_1^\prime, b_3^\prime\}$.  This yields
\begin{equation*}
e(\mathcal{G^{\prime}}| \mathcal{X})\geq\left\{
\begin{aligned}
&(m+2) (v(\mathcal{G^{\prime}}| \mathcal{X})-2)+4,\ \text{if}\ \mathcal{X}=\emptyset, \\
&(m+2)(m+1)/2,\ \text{if}\ \mathcal{X}=V(G^{\prime})\setminus \{A, b_1^\prime, b_3^\prime\},
\end{aligned}
\right.
\end{equation*}
a contradiction. So this corollary holds.
\end{proof}

\section{Proof of Theorem \ref{theorem8}}\label{section3}
We will prove Theorem \ref{theorem8} by induction on $m$. The main induction step will be given in  this section.

 \begin{lemma}\label{lemma2}
Let $(G, \{a_1,a_2\},b_1,b_2)$ be a 2-rooted graph with $v(G)\geq 5$. If $a_1$ and $a_2$ have no common neighbour in $G-\{b_1,b_2\}$, then $G$ has a $\{a_1, a_2,b_1,b_2\}$-collection $\mathcal{X}$ such that  $|N_G(X)| \leq 4$ for all $X\in \mathcal{X}$, and $e(\mathcal{G}| \mathcal{X})\leq 4v(\mathcal{G}| \mathcal{X})-11$.
\end{lemma}
\begin{proof}
On the contrary, suppose this lemma fails. Let $G^\prime$ be a graph obtained from $G$ by deleting all edges connecting $a_1,a_2$, and contracting $a_1,a_2$ to a vertex $a^*$. Write  $\mathcal{G^\prime}=(G^\prime, \{a^*,u\},b_1,b_2)$, where $u\in  V(G)\setminus \{a_1,a_2,b_1,b_2\}$ ($u$ exists since  $v(G)\geq 5$). According to Theorem  \ref{theorem1}, either $\mathcal{G^{\prime}}$ is feasible, or there is some $\{a^*, u, b_1, b_2\}$-collection $\mathcal{X}$ in $G^{\prime}$ such that $|N_{G^{\prime}}(X)| \leq 3$ for all $X\in \mathcal{X}$ and $(G^{\prime}/\mathcal{X} , a^*, b_1, u, b_2)$ is planar. By the definition and construction, $\mathcal{G^{\prime}}$ is not feasible. So $\mathcal{X}$  is a $\{a_1, a_2, b_1, b_2\}$-collection in $G$ such that $|N_{G}(X)|\leq |N_{G^{\prime}}(X)|+1 \leq 4$ for all $X\in \mathcal{X}$, and
\begin{align*}
E(\mathcal{G}| \mathcal{X})=&
\{a_1w: w\in N_{\mathcal{G}|\mathcal{X}}(a_1)\setminus \{a_2,b_1,b_2\}\} \cup \{a_2w: w\in N_{\mathcal{G}| \mathcal{X}}(a_2)\setminus \{a_1,b_1,b_2\}\}\\
&\cup E(G^\prime/\mathcal{X})\cup \{a_1a_2,a_1b_1,a_1b_2,a_2b_1,a_2b_2\}.
\end{align*}

Because $a_1$ and $a_2$ have no common neighbour in $G-\{b_1,b_2\}$, it follows that
$$|\{a_1w: w\in N_{\mathcal{G}\mid \mathcal{X}}(a_1)\setminus \{a_2,b_1,b_2,u\}\} \cup \{a_2w: w\in N_{\mathcal{G}\mid \mathcal{X}}(a_2)\setminus \{a_1,b_1,b_2,u\}\}|\leq v(\mathcal{G}| \mathcal{X})-5,$$
and one may assume $a_1u\notin E(G)$. Let $G^{\prime\prime}$ be a graph obtained from $G^{\prime}$ by  replacing $a^*$ by $a_1$, and deleting edges $a^*u$, $a_1w$ with $w\in N_{\mathcal{G}\mid \mathcal{X}}(a_1)\setminus \{a_2,b_1,b_2\}$. Further, we can draw edges $a_1b_1$, $a_1b_2$, $a_1a_2$, $a_2b_1$, $a_2b_2$ and $a_2u$ outside the disk, without introducing edge crossings. So $(G^{\prime\prime}/\mathcal{X}) \cup \{a_1a_2,a_1b_1,a_1b_2,a_2b_1,a_2b_2\}$ is planar. By Euler formula,
$$ e((G^{\prime\prime}/\mathcal{X}) \cup \{a_1a_2,a_1b_1,a_1b_2,a_2b_1,a_2b_2,a_2u\}) \leq 3 v(G/ \mathcal{X})-6.$$
This implies $e(\mathcal{G}|\mathcal{X}) \leq 3 v(G/ \mathcal{X})-6+ v(G/ \mathcal{X})-5= 4v(\mathcal{G}| \mathcal{X})-11$, as required.
\end{proof}

\begin{lemma}\label{lemma4}
Theorem \ref{theorem8} holds for $m = 0, 1, 2$.
\end{lemma}
\begin{proof}
 Let $\mathcal{G} = (G, \{a_1, \ldots , a_m\}, b_1, b_2)$ be an $m$-rooted graph. We first prove that Theorem \ref{theorem8} (ii) holds for $m = 0, 1, 2$. If not, then $m=0,1$ by Lemma \ref{lemma2} and the condition of Theorem \ref{theorem8}, and $G$ has no two internally $b_1$-$b_2$ paths in $G-\{a_m\}$. So, there exists a vertex $u$ in $G-\{b_1,b_2,a_m\}$ such that $G-\{u,a_m\}$ has no $b_1$-$b_2$ path ($u$ exists since $v(G)\geq m+3$). Let $D$ be the component of $G-\{u,a_m\}$ containing $b_1$; so $b_2 \notin V(D)$. Write $\mathcal{X} = \{V(D)-\{b_1\}, V (G-D)-\{u,b_2,a_m\}\}$. Then $|N_G(X)| \leq m+2$ for all $X\in \mathcal{X}$ and
\begin{equation*}
e(\mathcal{G}|\mathcal{X})\leq\left\{
\begin{aligned}
&2=2\cdot v(G|\mathcal{X})-4,\ \text{if}\ m=0, \\
&5=3\cdot v(G|\mathcal{X})-7,\ \text{if}\ m=1.
\end{aligned}
\right.
\end{equation*}
This contradicts the condition of Theorem \ref{theorem8}. So Theorem \ref{theorem8} (ii) holds for $m = 0, 1, 2$. It remains to  show  that Theorem \ref{theorem8} (i) also holds for $m = 0, 1, 2$.

On the contrary, suppose that $\mathcal{G}$ does not satisfy Theorem \ref{theorem8} (i). For $m=0,1$, $G$ has no $b_1$-$b_2$ path in $G-\{a_m\}$.  Let $D$ be the component of $G-\{a_m\}$ containing $b_1$; so $b_2 \notin V(D)$. Let $\mathcal{X} = \{V(D)-\{b_1\}, V (G-D)-\{a_m,b_2\}\}$. Then $|N_G(X)| \leq  m+2$ for all $X\in \mathcal{X}$ and
\begin{equation*}
e(\mathcal{G}| \mathcal{X})\leq\left\{
\begin{aligned}
&0=2\cdot v(G\mid \mathcal{X})-4,\ \text{if}\ m=0, \\
&2=3\cdot v(G\mid \mathcal{X})-7,\ \text{if}\ m=1.
\end{aligned}
\right.
\end{equation*}
This contradicts the condition of Theorem \ref{theorem8}. So Theorem \ref{theorem8} (i) holds for $m = 0, 1$.

For $m=2$, Lemma \ref{lemma2} and the condition of Theorem \ref{theorem8} state that $a_1$ and $a_2$ have a common neighbour $v$ in $G-\{b_1,b_2\}$. By our assumption, one may assume $vb_1\notin E(G)$. Let $G^{\prime}=G-\{v\}$ and $\mathcal{G^{\prime}}=(G^{\prime}, \{a_1,a_2\},b_1,b_2)$.  According to Theorem  \ref{theorem1}, either $\mathcal{G^{\prime}}$ is feasible, or there is some $\{a_1, a_2, b_1, b_2\}$-collection $\mathcal{X}$ in $G^{\prime}$ such that $|N_{G^{\prime}}(X)| \leq 3$ for all $X\in \mathcal{X}$ and $(G^{\prime}/\mathcal{X} , a_1, b_1, a_2, b_2)$ is planar. If the former holds, then $G$ is 2-feasible. If the latter holds, then $|N_{G}(X)|\leq |N_{G^{\prime}}(X)|+1 \leq 4$ for all $X\in \mathcal{X}$, and we can draw edges  $a_1b_1$, $a_1b_2$, $a_1a_2$, $a_2b_1$, $a_2b_2$, $a_2v$, $a_1v$ and $b_2v$  such that  $(\mathcal{G}^{\prime}/\mathcal{X}) \cup \{a_2v,a_1v, b_2v\}$ is planar. It follows from Euler formula that $e((\mathcal{G}^{\prime}/\mathcal{X})\cup \{a_2v,a_1v, b_2v\})\leq 3v(G/\mathcal{X})-6$.
Then
\begin{align*}
e(\mathcal{G}| \mathcal{X}) \leq e(\mathcal{G}/\mathcal{X})\leq e(\mathcal{G^{\prime}/X})+d_{\mathcal{G}/\mathcal{X}}(v)\leq 3v(G/\mathcal{X})-6+v(G/\mathcal{X})-5=4v(G|\mathcal{X})-11.
   \end{align*}
This contradicts the condition of Theorem \ref{theorem8} again.  This completes the proof of this lemma.
\end{proof}

 We now proceed by induction on $m$. Owing to Lemma \ref{lemma4}, Theorem \ref{theorem8} holds for $m = 0, 1, 2$. Suppose $m \geq 3$, Theorem \ref{theorem8} hold for all $(m-1)$-rooted graphs. Let $G^{\prime}$ be the graph obtained from $G$ by deleting all edges between $a_1,a_2$ and $N_{G}(a_1)\cup N_{G}(a_{2}) \setminus (N_{G}(a_{1})\cap N_{G}(a_{2}))$, and contracting $a_1,a_2$ to a vertex $a^*$. Let $\mathcal{G^{\prime}}=(G^{\prime},\{a^*,a_3,\ldots,a_m\},b_1,b_2)$. We see that if $\mathcal{G^{\prime}}$ satisfies the condition of Theorem \ref{theorem8}, then by induction, the following two statements hold.

 (i) Either $a^*,a_3,\ldots,a_m, b_1,b_2$ have a common neighbour in $G- \{a^*,a_3,\ldots,a_m, b_1, b_2\}$, or $\mathcal{G^\prime}$ is $2$-feasible.

(ii)  Either $a^*,a_3,\ldots,a_m$ have a common neighbour in $G- \{a^*,a_3,\ldots,a_m, b_1, b_2\}$, or $\mathcal{G^\prime}$ is $(2,2)$-feasible.

\noindent By our construction, Theorem \ref{theorem8} holds. Note that $v(G^\prime)\geq m+2$, so $G^\prime$  has a $\{a^*,a_3,\ldots,a_m, b_1, b_2\}$-collection $\mathcal{X}$ such that $|N_{G^\prime}(X)| \leq m+1$ for all $X\in \mathcal{X}$, and
\begin{align*}
e(\mathcal{G}^\prime| \mathcal{X})&\leq (m+1)v(\mathcal{G}^\prime| \mathcal{X})-(m-1)^2/2-5(m-1)/2-4\\
&=(m+1)v(\mathcal{G} | \mathcal{X})-m^2/2-5m/2-3.
\end{align*}
%
It is easy to see that $\mathcal{X}$  is a $\{a_1,a_2,a_3,\ldots,a_m, b_1, b_2\}$-collection in $G$ such that $|N_{G}(X)|\leq |N_{G^\prime}(X)|+1\leq m+2$ for all $X\in \mathcal{X}$.  By the definition of $\mathcal{G}| \mathcal{X}$, we obtain
$$E(\mathcal{G}| \mathcal{X})=E(\mathcal{G^\prime}| \mathcal{X})\cup \{a_iw: w\in N_{\mathcal{G}| \mathcal{X}}(a_i)\cup \{a_3,\ldots,a_m,b_1,b_2\} \text{ with } i\in [2]\}.$$
Note that
\begin{align*}
&|\{a_iw: w\in N_{\mathcal{G}| \mathcal{X}}(a_i)\cup \{a_3,\ldots,a_m,b_1,b_2\} \text{ with } i\in [2]\}| \\
\leq &(v(\mathcal{G}| \mathcal{X})-1)+|\{a_jw: w\in N_{\mathcal{G}| \mathcal{X}}(a_1)\cap N_{\mathcal{G}| \mathcal{X}}(a_2) \text{ for some } j\in [2]\}|
\end{align*}
and
$$\{a_jw: w\in N_{\mathcal{G}| \mathcal{X}}(a_1)\cap N_{\mathcal{G}| \mathcal{X}}(a_2)\text{ for some } j\in [2]\} \subseteq E(\mathcal{G^\prime| X}).$$
This implies $e(\mathcal{G}| \mathcal{X})\leq e(\mathcal{G^\prime}| \mathcal{X})+v(\mathcal{G}| \mathcal{X})-1 \leq (m+2)v(\mathcal{G} | \mathcal{X})-m^2/2-5m/2-4$, a contradiction. So Theorem \ref{theorem8} holds.

\section{Proof of Theorems \ref{theorem7} and \ref{theorem3}}\label{section4}
In this section, we will show Theorems \ref{theorem7} and \ref{theorem3}. We first give a definition. For two sequences $(a_1,\cdots,a_l)$ and $(b_1,\cdots,b_{l^\prime})$ with $l< l^\prime$ and $a_i$ = $b_i$ for any  $i \in [l]$, we regard that $(b_1,\cdots,b_{l^\prime})$ is larger than $(a_1,\cdots,a_l)$ \emph{in lexicographic order}.

In what follows, let $G$ be a  graph and  $\{a_1,\ldots,a_m\} \subseteq V(G)$, where $m$ is a positive integer. Write $A=\{a_1,\ldots,a_m\}$.  Now we show the following result holds.
\begin{lemma}\label{lemma1}
Let $G$ be a $(2m+6)$-connected graph and let $e_1,e_2$ be two distinct edges of $G-A$. If there exists a cycle $C$ in $G$ such that $\{e_1,e_2\}\subseteq E(C)$ and $G-V(C)$ has a block (resp. component) $B$ containing  $A$, then there exists a cycle $C^\prime$ in $G-A$ such that $\{e_1,e_2\}\subseteq E(C^\prime)$ and $G-V(C^\prime)$ is 2-connected (resp. connected).
\end{lemma}
\begin{proof}
On the contrary, suppose this lemma fails. Let  $B_1,\ldots, B_t$ denote the components of $G-V(C\cup B)$ such that $v(B_1)\geq \cdots \geq v(B_t)$.  We choose $C$ and $B$ such that
$$(v(B), v(B_1), \cdots ,v(B_t)) \text{ is as large as possible in lexicographic order.}$$

By deleting $e_1,e_2$, we have two disjoint paths, which denoted by $P_1$ and $P_2$. Further, for each $i\in [2]$, assume that $P_i$ is a path from  $b_i$ to $b_i^\prime$. By our choice, the following result holds.

\begin{claim}\label{claim3}
Suppose that $u_1, v_1, u_2,v_2$ are four distinct vertices and $u_i, v_i$ occur along the orientation of $P_i$ for each $i\in [2]$. If $\{u_1v_2, u_2v_2\}\subseteq E(G)$, then $u_iv_i\in E(P_i)$ for each $i\in [2]$. Furthermore, $P_1,P_2$ are induced.
\end{claim}

 For each $i\in [2]$,  let $c_i, d_i \in  N_{P_i}(B_t)$ with $P_i[c_i, d_i]$ maximal. Next we will perform
the following operation, and we shall update the vertices $c_i$ and $d_i$ for some  $i\in [2]$ at each step.\\
\textbf{Operation 1.} For each $i\in [2]$, we define
\begin{equation*}
Q_i=\left\{
\begin{aligned}
&P_i[b_i,c_i) \cup P_i(d_i,b_i^\prime],\ \text{if}\ c_i \text{ and } d_i \text{ exist}, \\
&P_i,\  \text{otherwise}.
\end{aligned}
\right.
\end{equation*}
Suppose that there exists an edge $uv$ connecting $Q_i$ and $P_j(c_j,d_j)$ for $\{i,j\}=\{1,2\}$ with $u \in V(Q_i)$ and $v \in V(P_j(c_j,d_j))$. If $u\in V(P_i[b_i,c_i))$, then we regard $u$ as the
new $c_i$, and if $u\in V (P_i(d_i,b_i^\prime])$, then we regard $u$ as the new $d_i$. If $c_i$ and $d_i$ do not exist, then let
$c_i = d_i = u$ as the new $c_i$ and $d_i$.

We perform Operation 1 as many times as possible. Note that it must stop, since  $V(Q_1 \cup Q_2)$ becomes smaller in each step. For the last $c_i$'s and $d_i$'s, let $C= \{c_i: i\in [2]
\text{ and } c_i \text{ exists} \}$ and $D= \{d_i: i\in [2]
\text{ and } d_i \text{ exists} \}$.
By the definition of $C$ and $D$, we obtain the following claim.

\begin{claim}\label{claim1}
 There exists no edge connecting $Q_1\cup Q_2$ and $ P_1(c_1,d_1) \cup P_2(c_2,d_2)\cup B_t$.
\end{claim}

 The following claim is crucial for our proof.

\begin{claim}\label{claim2}
The following statements hold.

(i) For any $j\in [t-1]$, there exists no edge connecting $B_j$ and $P_1(c_1,d_1) \cup P_2(c_2,d_2)$.

(ii) For any $i\in [2]$, there exists a vertex $z_i$ in $B$ such that there exist no edges connecting $B -\{z_i\}$ and $P_i(c_i,d_i)$.
\end{claim}
\begin{proof}
For each $i\in [2]$,  let $u_{i,1}, u_{i,2} \in  N_{P_i}(B_t)$ with $P_i[u_{i,1}, u_{i,2}]$ maximal. On the contrary, suppose that (i) and (ii) fail. Now we prove that there exists some $j\in [2]$ such that $u_{j,1}\in C$. If not, then Operation 1  yields that there exist two vertices $c_i^\prime\in P_i[b_i,u_{i,1})$ and $w_j\in P_j(u_{j,1}, b_j^\prime)$ such that $c_i^\prime w_j\in E(G)$ with $\{i,j\}=\{1,2\}$. One may assume $i=1$ and $j=2$. Let $c_2^\prime$ be the first vertex we get by updating $u_{2,1}$. Then $c_2^\prime\in P_2[b_2,u_{2,1})$ and  there exists $w_1\in P_1(c_1^\prime, b_1^\prime)$ such that $c_2^\prime w_1\in E(G)$, which contradicts Claim \ref{claim3}. Therefore, there exists some $j\in [2]$ such that $u_{j,1}\in C$. Similarly, there also exists some $j^\prime\in [2]$ such that $u_{j^\prime,2}\in D$. Then the proof will be divided into the following two cases.

\textbf{Case 1: $j\neq j^\prime$.}

One may assume $j=2$ and $j^\prime=1$. By our assumption and symmetry, assume $u_{1,1}\notin C$. It follows from  Operation 1 that $c_1$ has an neighbour $w_2$ in $P_2(u_{2,1},d_2)$. So, $G[V(B_t \cup C)-P_1(c_1,u_{1,2})]$ has a cycle $C^\prime$ such that $\{e_1,e_2\}\subseteq E(C^\prime)$. By our choice, this claim holds.


\textbf{Case 2: $j=j^\prime$.}

One may assume $j=j^\prime=1$. Then this claim holds for $i=1$.  It follows from Case 1 that $u_{2,1}\notin C$  and $u_{2,2}\notin D$. By Operation 1,  both $c_2$ and $d_2$  have an neighbour in $P_1(u_{1,1},u_{1,2})$. Suppose that $\{w_1,w_2\}\subseteq V(P_1(u_{1,1},u_{1,2}))$ and $\{c_2w_1,d_2w_2\}\subseteq E(G)$. Owing to Claim \ref{claim3}, $w_1,w_2$ occur along the orientation of $P_1$.  So, $G[V(B_t \cup C)-P_2(c_2,d_2)]$ has a cycle $C^\prime$ such that $\{e_1,e_2\}\subseteq E(C^\prime)$. By our choice, this claim holds for $i=2$.
\end{proof}

For each $i\in [2]$, we take the vertex $z_i$ as in Claim \ref{claim2} (ii), and let $z$ be a cut vertex of $G$ which
separates $B_t$ and $B$. By Claims \ref{claim1} and \ref{claim2}, $C \cup D \cup
\{z_1, z_2, z\}$ separates $P_1(c_1,d_1)\cup P_2(c_2,d_2) \cup  B_t$ from the other part (Possibly, $z_1 = z_2$  and $z\in \{z_1,z_2\}$). Combining this with the connectivity of $G$ implies
$$G-(C \cup D \cup
\{z_1, z_2, z\})-(P_1(c_1,d_1)\cup P_2(c_2,d_2) \cup  B_t)=\emptyset.$$ So $\{b_1,b_1^\prime\}\subseteq  C$ and $\{b_2,b_2^\prime\}\subseteq D$. By the connectivity of $G$, there exists an edge connecting $B_{j^\prime}$ and $P_1(b_1,b_1^\prime) \cup P_2(b_2,b_2^\prime)$ for some $j^\prime\in [t-1]$ or there are at least two edges connecting $B$ and $P_j(b_j,b_j^\prime)$ for some $j\in [2]$, which contradicts Claim \ref{claim2}. This proves this lemma.
\end{proof}

\begin{proof}
\noindent\textbf{of Theorem \ref{theorem3}.} Theorem \ref{theorem3} follows easily from Corollary \ref{corollary2} and Lemma \ref{lemma1}.
\end{proof}
\begin{proof}
\textbf{of Theorem \ref{theorem7}.} Let us first show $f(m,2,2)\leq 2m+6$.
Let $G$ be a $(2m+6)$-connected graph. Owing to Corollary \ref{corollary1} (ii), $(G,A,b_1,b_2)$ is (2,2)-feasible. Therefore, there exists a cycle $C$ containing  $b_1,b_2$, and $G-V(C)$ has a block containing $A$. By the connectivity of $G$ and Lemma \ref{lemma1}, this result holds.

  Now we will prove $f(m,1,2)\leq 2m+4$.  Let $G$ be a $(2m+4)$-connected graph and suppose that $G$ does not satisfy Theorem \ref{theorem7}. Owing to Corollary \ref{corollary1} (i), $(G,A,\{b_1,b_2\})$ is $2$-feasible. This implies that there is a $b_1$-$b_2$ path $P$ and $A$ is contained in a block $B$ of $G-V(P)$. Let $B_1,\ldots, B_t$ denote the components of $G-V(P\cup B)$ such that $v(B_1)\geq \cdots \geq v(B_t)$.  We choose $P$ and $B$ such that ($v(B), v(B_1), \ldots ,v(B_t)$) is as large as possible in lexicographic order.

It suffices to show that $t=0$. On the contrary, suppose $t \geq 1$. Since $G$ is $(2m+4)$-connected, there exist $2m+4$ disjoint paths from $B_t$ to $B,B_1,\ldots,B_{t-1}$. Let $u_1, u_2 \in  N_P(B_t)$
with $P[u_1, u_2]$ maximal, and either $u_1^\prime, u_2^\prime \in  N_{P(u_1,u_2)}(B)$ or  $u_1^\prime\in  N_{P(u_1,u_2)}(B_j)$ for some $j\in [t-1]$. Then we get a contradiction against our choice. This completes the proof.
\end{proof}

\end{document}